\documentclass[graybox]{svmult}

\usepackage{mathptmx}       
\usepackage{helvet}         
\usepackage{courier}        
\usepackage{type1cm}        
\usepackage{makeidx}         
\usepackage{graphicx}        
\usepackage{multicol}        
\usepackage[bottom]{footmisc}

\usepackage{amsfonts}

\def\l{\left}
\def\r{\right}

\def\R#1{$(\ref{#1})$}
\def\qed{\hspace{3mm}\hbox{\vrule height7pt width5pt}}

\newcommand{\bb}[1]{\begin{equation}\label{#1}}
\newcommand{\ee}{\end{equation}}
\newcommand{\bbb}{\begin{eqnarray}}
\newcommand{\eee}{\end{eqnarray}}
\newcommand{\bbbb}{\begin{eqnarray*}}
\newcommand{\eeee}{\end{eqnarray*}}

\newcommand{\nnn}{\nonumber}

\definecolor{green1}{rgb}{0.1,0.5,0.0}

\spnewtheorem{assume}{Assumption}{\bfseries}{\itshape}

\usepackage[numbers,sort&compress]{natbib}

\makeindex             

\begin{document}

\title*{Convergence of an operator splitting scheme for abstract stochastic evolution equations}
\author{Joshua L. Padgett and Qin Sheng}
\institute{Joshua L. Padgett \at Department of Mathematics and Statistics, Texas Tech University, Lubbock, TX 79409-1042, USA \email{joshua.padgett@ttu.edu}
\and Qin Sheng \at Department of Mathematics and
Center for Astrophysics, Space Physics and Engineering Research,
Baylor University, Waco, TX 76798-7328, USA \email{qin\_sheng@baylor.edu}}
%
%
\maketitle

\abstract*{
In this paper we study the convergence of a Lie-Trotter operator splitting for stochastic semilinear evolution 
equations in a Hilbert space. The abstract Hilbert space setting allows for the consideration of convergence 
of the approximation for both the original and spatially discretized problems. It is known that the strong 
convergence of this scheme is classically of half-order, at best. We demonstrate that this is in fact the optimal order 
of convergence in the proposed setting, with the actual order being dependent upon the regularity of noise
collected from applications.
}

\abstract{
In this paper we study the convergence of a Lie-Trotter operator splitting for stochastic semi-linear evolution 
equations in a Hilbert space. The abstract Hilbert space setting allows for the consideration of convergence 
of the approximation for both the original and spatially discretized problems. It is known that the strong 
convergence of this scheme is classically of half-order, at best. We demonstrate that this is in fact the optimal order 
of convergence in the proposed setting, with the actual order being dependent upon the regularity of noise
collected from applications.
}


\section{Introduction}
\label{sec:1}

Geometric integration techniques have received much attention in the study of differential 
equations \cite{hairer2006geometric,Blanes_book,Iserles1,Malham1}. In particular, operator 
splitting methods have been shown to be effective and efficient numerical methods, as they 
may often be constructed to preserve stability while being explicit with desirable convergence rates \cite{Hansen2008,Hansen2012,Josh1,Josh2,Sheng1989,Sheng1994}. While splitting methods 
have primarily been studied in the deterministic setting, there have been several recent studies regarding 
their efficacy in application to stochastic problems \cite{Josh3,misawa2000numerical,Burrage1,Misawa1}. 
In particular, it has been shown that the splitting of deterministic and stochastic counterparts of differential 
equations can prove effective by increasing convergence rates without the inclusion of derivative terms \cite{misawa2000numerical,Burrage1,cox2010convergence}. Moreover, it is known that operator splitting 
methods may preserve many desirable geometric properties of the true solution, including the 
monotonicity and positivity \cite{Hansen2012,iserles_2008,Josh3}.

Due to its wide range of applications in sciences and engineering, this article considers the following 
semi-linear stochastic differential equation problem,
\bbb
&& du = \l[Au + f(u)\r]dt + g(u)\,dW,\quad 0\le t\le T,\label{e1}\\
&& u(0) = u_0 \in H,\label{e2}
\eee
where $H$ is a separable Hilbert space. 
In the above, $A\,:\,\mbox{Dom}(A)\subset H \to H$ is a linear operator whose domain is dense in $H$ and compactly 
embedded into $H.$ We will further assume that $A$ generates an analytic semigroup $e^{tA},\ t\ge 0.$ The operators 
$f$ and $g$ are assumed to be Lipschitz continuous and possess continuous, uniformly bounded Fr{\'e}chet derivatives up to order two. 
These assumptions, and the precise analytic framework for \R{e1}-\R{e2} 
will be further outlined in Section 2. For technical reasons, we assume $u_0\in H$ to be deterministic.


Without loss of generality, we let $N\in \mathbb{N}$ be fixed, and define $h = 1/N.$ We are concerned with developing 
an approximation to the true solution to \R{e1}-\R{e2} at time $t_n=nh,$ denoted $u_n,$ being given by
\bb{s1}
u_n = S^n(u_0),
\ee
where $S\,:\,H\to H$ is the nonlinear operator defined as
\bb{s2}
S \mathrel{\mathop:}= e^{hA}e^{hf}e^{\Delta W(h)g}.
\ee
The nonlinear operator $v(t) = e^{hf}(v_0)$ is the solution to the differential equation $dv = f(v)\,dt$ at time $t$ with initial 
condition $v(0) = v_0,$ while $z(t) = e^{\Delta W(h)g}(z_0)$ is the solution to the stochastic differential equation 
$dz = g(z)\,dW$ at time $t$ with initial condition $z(0) = z_0.$ Such operators are often referred to as the nonlinear 
semigroup for each problem \cite{kato1967nonlinear}.

The splitting scheme given by \R{s1} and \R{s2} is classically known as the Lie-Trotter splitting scheme and has been 
well-studied in numerous settings \cite{Trotter1959,iserles_2008,Hansen2012,Jahnke2000}. Such methods have been 
studied in the finite-dimensional stochastic setting for ordinary differential equations via Lie algebraic techniques 
\cite{misawa2000numerical,Burrage1,Misawa1}. There has also been a recent study of such problems for linear 
equations with additive noise in UMD Banach spaces \cite{cox2010convergence}. In this study, the optimal convergence 
rate was recovered, while the effects of nonlinearities were not included.
However, the inclusion of nonlinear multiplicative noise terms complicates the required analysis and becomes one of the
 concerns of this current article.

This article is organized as follows. In Section 2, the abstract setting utilized throughout the article is detailed with several 
necessary results recalled. Section 3 outlines several basic properties regarding stability issues of the proposed operator splitting 
scheme. Section 4 is concerned with a detailed consistency analysis while Section 5 demonstrates the desired convergence result. 

\section{Abstract Stochastic Evolution Problems}
\label{sec:2}
Let $H$ be a separable Hilbert space with inner product $\langle \cdot,\cdot\rangle$ and associated norm $\|\cdot\| = \langle \cdot,\cdot\rangle^{1/2}.$ 
For another Hilbert space $U$ equipped with norm $\|\cdot\|_U,$ we denote by $L(U,H)$ the set of bounded linear operators from $U$ to $H.$
For the simplicity of of notations, we let $L(U,U) = L(U).$ Further, we denote by $L_1(U,H)$ the set of nuclear operators from $U$ to $H$ 
and $L_2(U,H)$ the set of Hilbert-Schmidt operators from $U$ to $H.$ Further, if $\{e_i\}_{i\in\mathbb{N}}$ forms an arbitrary orthonormal 
basis of $H,$ then we have the following norms associated with the aforementioned spaces:
$$
\|\Gamma\|_{L_1(U)} \mathrel{\mathop:}= \mbox{Tr}\l(\Gamma^*\Gamma\r)^{1/2} = \sum_{i = 1}^\infty \langle \l(\Gamma^*\Gamma\r)^{1/2}e_i,e_i\rangle,
~~\Gamma\in L_1(U),
$$
and 
$$
\|\Gamma\|_{L_2(U)}^2 \mathrel{\mathop:}= \sum_{i=1}^\infty \|\Gamma e_i\|_U^2,~~\Gamma\in L_2(U),
$$
where $\Gamma^*$ denotes the adjoint of $\Gamma.$ We further let $\mathbb{E}\|\cdot\|_{L_1(H)}$ and $\mathbb{E}\|\cdot\|_{L_2(H)}$ denote the corresponding
expected values of each norm. Moreover, the trace and Hilbert-Schmidt norms are independent of the given basis.

Let $(\Omega,\mathcal{F},\mathbb{P})$ be a probability space with normal filtration $\{\mathcal{F}(t)\}_{t\ge 0},$ and let $W(t)$ be a standard Wiener process with covariance operator $Q,$
where $Q\,:\,H\to H$ is a positive self-adjoint operator. 
If $q_i>0$ are the eigenvalues of $Q$ corresponding to eigenfunctions $e_i,$ $i\in\mathbb{N},$ we then have
$$W(t)\mathrel{\mathop:}= \sum_{i\in\mathbb{N}} \sqrt{q_i}\beta_i(t)e_i,\quad 0\le t\le T,$$
where $\{\beta_i\}_{i\in\mathbb{N}}$ are independent, real-valued Brownian motions on the probability space. 

We denote the set of Hilbert-Schmidt operators from $Q^{1/2}(H)$ to $H$ by $L_2^0(H)$ and its norm, for $\Gamma\in L_2^0(H),$ is given by
$$
\|\Gamma\|_{L_2^0(H)} \mathrel{\mathop:}= \|Q^{1/2}\Gamma\|_{L_2(H)} = \l(\sum_{i=1}^\infty \|Q^{1/2}(\Gamma^*\Gamma)^{1/2}e_i\|^2\r)^{1/2}.
$$
Now let $\varphi\,:[0,T]\times \Omega \to L_2^0(H)$ be an $L_2^0(H)-$valued predictable stochastic process with
$$\int_0^t \mathbb{E}\|Q^{1/2}\varphi\|_{L_2(H)}^2 \,ds <\infty,\quad 0\le t\le T,$$
then Ito's isometry (see, for instance, \cite{pz_2014}) gives
$$
\mathbb{E}\l\|\int_0^t \varphi \,dW\r\|^2 = \int_0^t \mathbb{E} \|\varphi\|_{L_2^0(H)}^2\,ds = \int_0^t \mathbb{E}\|Q^{1/2}\varphi\|^2_{L_2(H)}\,ds,\quad 0\le t\le T.
$$

We now recall some basic properties of Hilbert space operators that will be of interest throughout this work.
\begin{proposition}
Let $\Gamma, \Gamma_1, \Gamma_2$ be three operators in Hilbert spaces. Then we have the following results.
\begin{itemize}
\item[i.] If $\Gamma\in L_1(U),$ then
$$
|\mbox{Tr}(\Gamma)|\le \|\Gamma\|_{L_1(U)}.
$$
\item[ii.] If $\Gamma_1\in L(U)$ and $\Gamma_2\in L_1(U),$ then both $\Gamma_1\Gamma_2$ and $\Gamma_2\Gamma_1$ belong to $L_1(U)$ with
$$
\mbox{Tr}(\Gamma_1\Gamma_2) = \mbox{Tr}(\Gamma_2\Gamma_1).
$$
\item[iii.] If $\Gamma_1\in L(U,H)$ and $\Gamma_2\in L(H,U),$ then $\Gamma_1\Gamma_2\in L_1(H)$ with
$$
\|\Gamma_1\Gamma_2\|_{L_1(H)} \le \|\Gamma_1\|_{L_2(U,H)}\|\Gamma_2\|_{L_2(H,U)}.
$$
\item[iv.] If $\Gamma\in L_2(U,H),$ then $\Gamma^*\in L_2(H,U)$ with
$$
\|\Gamma^*\|_{L_2(H,U)} = \|\Gamma\|_{L_2(U,H)}.
$$
\item[v.] If $\Gamma\in L(U,H)$ and $\Gamma_1,\Gamma_2\in L_i(U),\ i=1,2,$ then
$\Gamma\Gamma_1,\Gamma\Gamma_2\in L_i(U,H),\ i=1,2,$ with
$$
\|\Gamma\Gamma_i\|_{L_j(U,H)} \le \|\Gamma\|_{L(U,H)}\|\Gamma_i\|_{L_j(H)},\quad i=1,2,\ j=1,2.
$$
\end{itemize}
\end{proposition}
More details on the proposition and the spaces used can be found in \cite{chow2014stochastic,prevot2007concise}.






We now outline several assumptions necessary for the existence, uniqueness, and well-poseness of the 
solution to \R{e1}-\R{e2}. 
\begin{assume}
The linear operator $A\,:\,\mbox{Dom}(A)\subset H\to H$ is the generator of a bounded $C_0$ semigroup $e^{tA},\ t\ge 0.$
\end{assume}
Without loss of generality, by Assumption 1, it follows that we may assume that
$$\|e^{tA}\| \le 1,\quad t\ge 0.$$
We now outline some basic properties of the semigroup generated by $A$ (see, for instance, \cite{henry2006geometric}).

\begin{proposition}
Let $\alpha\ge 0$ and $0\le \gamma\le 1.$ Then there exists a constant $C>0$ such that
\begin{itemize}
\item[i.] $\|(-A)^\alpha e^{tA}\|_{L(H)} \le Ct^{-\alpha},$ for $t>0,$
\item[ii.] $(-A)^\alpha e^{tA} = e^{tA}(-A)^\alpha,$ on $\mbox{Dom}((-A)^\alpha),$
\item[iii.] If $\alpha\ge \gamma,$ then $\mbox{Dom}((-A)^\alpha) \subset \mbox{Dom}((-A)^\gamma).$
\end{itemize}
\end{proposition}

Recall \R{e1}. For nonlinear terms $f$ and $g,$ we need following restrictions. 
\begin{assume}
For the drift term $f\,:\,H\to H,$ assume that there exists a positive constant $L_f>0$ such that $f$ satisfies the following Lipschitz condition
$$
\|f(u)-f(v)\| \le L_f\|u-v\|,\quad \mbox{for all}\ u,v\in H.
$$
This yields the following growth condition:
$$
\|f(u)\| \le C(1+\|u\|),\quad \mbox{for all}\ u\in H.
$$
We further assume that the derivatives $Df[u]\,:\,H\to H$ and $D^2f[u]\,:\,H\times H\to H$ are continuous and uniformly bounded for all $u\in H.$ 
\end{assume}

\begin{assume}
For the diffusion term $g\,:\,H\to L_2^0(H),$ assume that there exists a positive constant $L_g>0$ such that $g$ 
satisfies the following Lipschitz condition
$$
\|g(u)-g(v)\|_{L_2^0(H)} \le L_g\|u-v\|,\quad \mbox{for all}\ u,v\in H.
$$
Similarly, the above leads to the growth condition:
$$
\|g(u)\|_{L_2(H)} \le C(1+\|u\|),\quad \mbox{for all}\ u\in H.
$$
We further assume that the derivatives $Dg[u]\,:\,H\to L_2^0(H)$ and $D^2g[u]\,:\,H\times H\to L_2^0(H)$ are continuous and uniformly bounded for all $u\in H.$
\end{assume}

In order to guarantee the existence of a well-defined mild solution to \R{e1}-\R{e2}, we must also invoke a standard regularity assumption on the covariance operator of the noise $W.$
\begin{assume}
Assume that there exists $\beta \in (0,1]$ and $C>0$ such that
\bb{normcond}
\l\|(-A)^{(\beta -1)/2}Q^{1/2}\r\|_{L_2(H)} = \l\|Q^{1/2}(-A)^{(\beta -1)/2}\r\|_{L_2(H)} \le C.
\ee
\end{assume}
In the following analysis, any reference to a parameter $\beta$ is the same $\beta$ defined in \R{normcond}.

If Assumptions 1-4 are satisfied and $u_0\in H$ is $\mathcal{F}_0-$measurable, then
it follows that \R{e1}-\R{e2} 
admits a unique (up to the equivalence of paths) mild solution $u\,:\,[0,T]\times \Omega \to H$ with continuous sample path given by
\bb{mildsol}
u(t) = e^{tA}u_0 + \int_0^t e^{(t-s)A}f(u(s))\,ds + \int_0^t e^{(t-s)A}g(u(s))\,dW(s),\quad \mathbb{P}-a.s.,
\ee
with the expectation
\bb{mildsolnorm}
\mathbb{E}\|u(t)\|^2 <\infty,\quad 0\le t\le T,
\ee
(see \cite{pz_2014}).

Let the Banach space $\mbox{Dom}((-A)^{\alpha/2}), \alpha\ge 0,$ be equipped with the standard norm given 
by $\|\cdot\|_\alpha \mathrel{\mathop:}= \|(-A)^{\alpha/2}\cdot\|.$ Then we have the following regularity result for the solution to \R{e1}-\R{e2} \cite{lord2012stochastic}. 

\begin{theorem}
Assume that Assumptions 1-4 hold. Let $u$ be the mild solution to \R{e1}-\R{e2} given by \R{mildsol}. 
If $u_0 \in L^2(\Omega,\mbox{Dom}((-A)^{\alpha/2})),\ \alpha\in[0,1),$ then for all $0\le t\le T,$  
$u\in L^2(\Omega,\mbox{Dom}((-A)^{\alpha/2}))$ and
$$\sup_{0\le t\le T}\l(\mathbb{E}\|u(t)\|^2_\alpha\r)^{1/2} \le C\l(1+\l(\mathbb{E}\|u_0\|_\alpha^2\r)^{1/2}\r).$$
\end{theorem}


In addition, we employ two more assumptions. 
\begin{assume}
We have $\mbox{Dom}(A)\subset H$ and $\mbox{Dom}(A^2)\subset H$ are both invariant under $f$ and $g,$ with 
$\mbox{Dom}(A)$ also being invariant under $Df$ and $Dg,$ for all $u\in \mbox{Dom}(A).$
\end{assume}

\begin{assume}
Let $\beta\in (0,1]$ be defined as in \R{normcond}. Then we assume that there exists a constant $C>0$ such that
$$\|(-A)^{(\beta-1)/2}Dg[\xi](u-v)\|_{L_2^0(H)} \le C\|u-v\|$$
and
$$\|(-A)^{(\beta-1)/2}D^2[\xi](u-v)^2\|_{L_2^0(H)} \le C\|u-v\|,$$
for all $\xi,u,v\in H.$
\end{assume}

Assumption 6 initially appears to be restrictive. However, since $\beta\in (0,1],$ the assumption 
actually allows for the derivatives $Dg$ and $D^2g$ to be slightly less regular.







Throughout this article, we will denote function and operator composition by left multiplication. That is, for two operators 
$F_1$ and $F_2,$ we use the standard notation
$$F_1F_2(u) = F_1(F_2(u)),$$
whenever the composition in consideration is well-defined. Furthermore, throughout this article, we let $C>0$ represent 
a generic constant independent of $n$ and $h.$ Note that this constant may assume different values throughout arguments.

In order to avoid repetition, it is henceforth assumed that Assumptions 1-6 hold throughout the remainder of the article. 
It is worth noting that Assumption 4 is quite standard and allows for the consideration of both space-time and trace class white noise. Space-time white noise corresponds to $Q = I$ and it is known that \R{normcond} is satisfied when $\beta<1/2,$ in the case of one spatial dimension. When considering trace class noise, that is when $\mbox{Tr}(Q)<\infty,$ it follows that \R{normcond} is satisfied for $\beta=1$ \cite{debussche2011weak}. By considering trace class noise, we are able to recover the results presented in \cite{misawa2000numerical,Misawa1}. 



\section{Properties of the Splitting Operator}
\label{sec:3}

We first define the {\em least upper bound (lub) Lipschitz constant} and {\em (lub) logartihmic Lipschitz constant} of a 
function $F\,:\,H\to H$ by
$$
L[F]\mathrel{\mathop:}= \sup_{u\neq v}\frac{\|F(u)-F(v)\|}{\|u-v\|}
$$
and
$$
M[F] \mathrel{\mathop:}= \lim_{k\to 0^+} \frac{L[I+kF]-1}{k},
$$
respectively. For the following lemmas, we will consider the following problems:
\bb{pf1}
dv = f(v)\,dt,\quad v(0) = v_0,
\ee
and
\bb{ppf1}
dv = g(v)\,dW,\quad v(0) = v_0.
\ee

\begin{lemma}
Let $v(t) = e^{hf}(v_0)$ be the solution to \R{pf1}. It then follows that
$$L[e^{hf}] \le e^{hL_f}.$$
\end{lemma}

\begin{proof}
Let $v$ and $w$ be two distinct solutions to \R{pf1}. Let $D_t^+$ denote the upper-right Dini derivative. 
Then, due to the assumptions on $f$ and its derivatives, we have
\bbbb
D_t^+\|v-w\| & = & \limsup_{h\to 0^+}\frac{\|v(t+h)-w(t+h)\|-\|v(t)-w(t)\|}{h}\nnn\\
& \le & \lim_{h\to 0^+} \frac{\|(I+hf)(v(t)-w(t))\|-\|v(t)-w(t)\|}{h}\nnn\\
& \le & \lim_{h\to 0^+}\frac{L[I+hf]\|v(t)-w(t)\|-\|v(t)-w(t)\|}{h}\nnn\\
& \le & M[f]\|v(t)-w(t)\|.\label{pf2}
\eeee
Solving the above inequality yields
$$
\|v(t)-w(t)\| \le e^{hM[f]}\|v_0-w_0\|.
$$
By the fact that $f$ is Lipschitz continuous in $H,$ we have
$$
M[f] \le L_f.
$$
This yields the desired result. \qed
\end{proof}

For the following lemma, we mirror the approach employed in Lemma 1, but we need to consider slightly modified Lipschitz constants. 
To that end, we define the {\em lub stochastic Lipschitz constant} and {\em lub logarithmic stochastic Lipschitz constant} of a function 
$G\,:\, H\to L_2^0(H)$ by
$$\mathbb{E}\l[L_2[G]\r]\mathrel{\mathop:}= \sup_{u\neq v}\frac{\mathbb{E}\|G(u)-G(v)\|_{L_2^0(H)}^2}{\|u-v\|^2}$$
and
$$M_2[G]\mathrel{\mathop:}= \lim_{h\to 0}\frac{L_2[I+hG]-1}{h},$$
respectively. 

\begin{lemma}
Let $v(t) = e^{\Delta W(h)g}(v_0)$ be the solution \R{ppf1}. It then follows that
$$\mathbb{E}\l[L_2[e^{\Delta W(h)g}]\r] \le e^{hL_g^2}.$$
\end{lemma}

\begin{proof}
We proceed in a fashion similar to that of the previous proof. Let $v$ and $w$ be two distinct solution to 
\R{ppf1} and let $D_t^+$ denote the upper-right Dini derivative. Hence, we have 
\bbbb
D_t^+ \mathbb{E}\|v-w\|^2 & = & \limsup_{h\to 0^+}\frac{\mathbb{E}\l[\|v(t+h)-w(t+h)\|^2 - \|v(t)-w(t)\|^2\r]}{h}\\
& \le & \lim_{h\to 0^+}\frac{\mathbb{E}\l[\|(I+hg)(v(t)-w(t))\|^2_{L_2^0(H)} - \|v(t)-w(t)\|^2\r]}{h}\\
& \le & M_2[g]\mathbb{E}\|v(t)-w(t)\|^2.
\eeee
Note that deriving the second inequality follows from the fact that the remainder terms from the expansion are bounded. The details of this claim can be found in the proofs of Lemmas 4 and 5.
Due to the expectation, the above inequality is deterministic and its solution is given by
$$\mathbb{E}\|v(t)-w(t)\|^2 \le e^{hM_2[g]}\mathbb{E}\|v_0-w_0\|^2.$$
Once again, since $g$ is Lipschitz in $H,$ we have
$$M_2[g] = \lim_{h\to 0}\frac{1}{h}\l[\sup_{u\neq v}\frac{\mathbb{E}\|(1+hg)(u-v)\|_{L_2^0(H)}^2}{\|u-v\|^2} - 1\r] \le L_g^2.$$
This yields the desired result. \qed
\end{proof}

\begin{lemma}
Consider \R{s1} and \R{s2}. Then we have
$$\mathbb{E}\|S^n(u)-S^n(v)\|^2 \le C\mathbb{E}\|u-v\|^2,$$
and in particular,
$$\mathbb{E}\l[L_2[S]\r] \le e^{hC}.$$
\end{lemma}

\begin{proof}
By Lemmas 1 and 2, we readily have the following estimates: 
\bbb
\mathbb{E}\|S(u)-S(v)\|^2 
& \le & \mathbb{E}\l[L[e^{hf}]^2\|e^{\Delta W(h)g}(u-v)\|_{L_2^0(H)}^2\r]\nnn\\
& \le & e^{2hL_f}\mathbb{E}\l[L_2[e^{\Delta W(h)g}]\|(u-v)\|^2\r]\nnn\\
& \le & e^{hC}\mathbb{E}\|u-v\|^2.\label{lipbound}
\eee
Via iterations, it follows immediately that
$$
\mathbb{E}\|S^n(u)-S^n(v)\|^2 \le \prod_{j=0}^n e^{hC}\mathbb{E}\|u-v\|^2 \le e^{TC}\mathbb{E}\|u-v\|^2,
$$
which gives the desired result. \qed
\end{proof}

\section{Approximation Consistency}
\label{sec:4}
Similar to discussions in \cite{Hansen2012} and \cite{Jahnke2000},  
we define
$$\phi(t)\mathrel{\mathop:}= \frac{1}{t}\int_0^t e^{(t-s)A}f(T(u(s)))\,ds,\quad \psi(t)\mathrel{\mathop:} = \frac{1}{t}\int_0^t e^{(t-s)A}g(T(u(s)))\,dW(s),$$
where $T(u)$ is the solution operator for \R{e1}-\R{e2}, and
\bb{tt}
T(u) = e^{hA}u + \int_0^h e^{(h-s)A}f(u)\,ds + \int_0^h e^{(h-s)A}g(u)\,dW(s),\quad \mathbb{P}-a.s.
\ee
Note that the operators are well-defined and map $H$ into itself for $u\in\mbox{Dom}(A).$ 

\begin{lemma}
Assume that $u\in\mbox{Dom}(A).$ Then 
$$\mathbb{E}\|(T-S)(u)\|^2 \le Ch^{2+\beta},$$
where $\beta\in(0,1]$ is defined in \R{normcond}.
\end{lemma}

\begin{proof}
By appealing to the stochastic version of Taylor's theorem, we arrive at
\bbb
S(u) & = & 
e^{hA}e^{hf}e^{\Delta W(h)g}u\nnn\\
& = & 
e^{hA}e^{\Delta W(h)g}u + he^{hA}fe^{\Delta W(h)g}u + R_1(u)\nnn\\
& = & 
e^{hA}u + he^{hA}f\l(u + g(u)\Delta W(h) + R_2(u)\r) + e^{hA}g(u)\Delta W(h) + R_1(u) + R_2(u)\nnn\\
& = & e^{hA}u + he^{hA}f(u) + e^{hA}g(u)\Delta W(h) + R_1(u) + R_2(u) + R_3(u),\label{t1}
\eee
where
$$\Delta W(h)\mathrel{\mathop:}= \int_0^h dW(s),$$
\bb{r1}
R_1(u) \mathrel{\mathop:}= 
h^2\int_0^1 (1-s)Df[e^{shf}e^{\Delta W(h)g}u]fe^{shf}e^{\Delta W(h)g}u\,ds,
\ee
\bbb
R_2(u) &\mathrel{\mathop:}=& 
\int_0^h\int_0^1 (1-y)e^{hA}D^2g[u+y(e^{\Delta W(s)g}u - u)](e^{\Delta W(s)g}u - u)^2\,dy\,dW(s)\nnn\\
&&~~~~~~~~~~ 
+ \int_0^h e^{hA}Dg[u](e^{\Delta W(s)g}u-u)\,dW(s),\label{r2}
\eee
and
\bb{r4}
R_3(u)\mathrel{\mathop:}= 
h\int_0^1e^{hA} Df[u + y(g(u)\Delta W(h) + R_2(u))](g(u)\Delta W(h) + R_2(u))\,dy.
\ee

Recall \R{tt}. We observe that
\bbb
T(u) & = & e^{hA}u + \int_0^h e^{(h-s)A}f(u)\,ds + \int_0^h e^{(h-s)A}g(u)\,dW(s)\nnn\\
& = & e^{hA}u + \int_0^h e^{(h-s)A}f(e^{sA}u + s\phi(s) + s\psi(s))\,ds \nnn\\
&&~~~~~~~~~~~~~~~+ \int_0^h e^{(h-s)A}g(e^{sA}u + s\phi(s) + s\psi(s))\,dW(s)\nnn\\
& = & e^{hA}u + \int_0^h e^{(h-s)A}f(e^{sA}u)\,ds + \int_0^h e^{(h-s)A}g(e^{sA}u)\,dW(s) + R_T(u),\label{tt1}
\eee
where
\bbb
R_T(u) &\mathrel{\mathop:}=& \int_0^h\int_0^1 e^{(h-s)A}Df[\xi(y;s)](s\phi(s)+s\psi(s))\,dy\,ds\nnn\\
&&~~~~~~~~~~~~~~~ + \int_0^h\int_0^1 e^{(h-s)A}Dg[\xi(y;s)](s\phi(s)+s\psi(s))\,dy\,dW(s),\label{ttr}
\eee
where $\xi(y;s)\mathrel{\mathop:}= e^{sA}u+y(s\phi(s)+s\psi(s)) \in H.$ Combining the above yields
\bb{sub}
(T-S)(u) = \int_0^h [z_1(s)-z_1(0)]\,ds + \int_0^h [z_2(s)-z_2(0)]\,dW(s) + R(u),
\ee
where
$$z_1(s)\mathrel{\mathop:}= e^{(h-s)A}f(e^{sA}u),~~~ z_2(s)\mathrel{\mathop:}= e^{(h-s)A}g(e^{sA}u)$$
and
$$R(u)\mathrel{\mathop:}= (R_T-R_1-R_2-R_3)(u).$$
Upon further expansion of \R{sub}, we obtain
\bb{sub2}
(T-S)(u) = h\int_0^h z_1'(\xi_1)\,ds + h\int_0^h z_2'(\xi_2)\,dW(s) + R(u),
\ee
for some $\xi_1,\xi_2\in [0,h].$

Utilizing the above equalities and Ito's isometry, we acquire that
\bbb
\mathbb{E}\|(T-S)(u)\|^2 & \le & 3\mathbb{E}\l\|h\int_0^h z'(\xi_1)\,ds\r\|^2 + 
3\mathbb{E}\l\|h\int_0^h z'(\xi_2)\,dW(s)\r\|^2 + 3\mathbb{E}\l\|R(u)\r\|^2\nnn\\
& = & 3h^2\int_0^h \mathbb{E}\|z_1'(\xi_1)\|^2\,ds\nnn\\
&&~~~~~~~~~~~~~~~ + 3h^2\int_0^h\mathbb{E}\|Q^{1/2}z_2'(\xi_2)\|_{L_2(H)}^2\,ds + 3\mathbb{E}\|R(u)\|^2\label{fbnd}.
\eee
It now remains to estimate each of the integrals in \R{fbnd}. 
To this end, we observe that 
$$z_1'(\xi_1) = -Ae^{(h-\xi_1)A}f(e^{\xi_1A}u) + e^{(h-\xi_1)A}Df[e^{\xi_1A}u]e^{\xi_1A}Au.$$
When $u\in\mbox{Dom}(A),$ it follows immediately that $z_1'\in H$ and thus
$$\mathbb{E}\|z_1'(\xi_1)\|^2 \le C.$$
By the same token, we have
$$z_2'(\xi_2) = -Ae^{(h-\xi_2)A}g(e^{\xi_2A}u) + e^{(h-\xi_2)A}Dg[e^{\xi_2A}u]e^{\xi_2A}Au \in H.$$
Thus,
\bbb
&&\mathbb{E}\l\|Q^{1/2}\l[-Ae^{(h-\xi_2)A}g(e^{\xi_2A}u) + e^{(h-\xi_2)A}Dg[e^{\xi_2A}u]e^{\xi_2A}Au\r]\r\|_{L_2(H)}^2\nnn\\
&& ~~~~~~~~~~~~~~~~~~~~~~~~~~~ \le 2\mathbb{E}\l\|Q^{1/2}(-A)e^{(h-\xi_2)A}g(e^{\xi_2A}u)\r\|_{L_2(H)}^2\nnn\\
&&~~~~~~~~~~~~~~~~~~~~~~~~~~~~~~~~~~~~~~~~ + 2 \mathbb{E}\l\|Q^{1/2}e^{(h-\xi_2)A}Dg[e^{\xi_2A}u]e^{\xi_2A}Au\r\|_{L_2(H)}^2.\nnn
\eee
Considering the first quantity in the above inequality, we find that
\bbb
&&\l\|Q^{1/2}(-A)e^{(h-\xi_2)A}g(e^{\xi_2A}u)\r\|_{L_2(H)}^2\nnn\\ 
&&~~~~~ \le\l\|Q^{1/2}(-A)^{(\beta -1)/2}\r\|_{L_2(H)}^2\l\|(-A)^{-\beta/2}e^{(h-\xi_2)A}\r\|_{L(H)}^2\l\|g(e^{\xi_2A}u)\r\|_{L_2(H)}^2\nnn\\
&&~~~~~ \le Ch^\beta (1+\|u\|)^2,\nnn
\eee
and thus by Theorem 1 and Lemma 3, we have
\bb{bnd11}
\mathbb{E}\l\|Q^{1/2}(-A)e^{(h-\xi_2)A}g(e^{\xi_2A}u)\r\|_{L_2(H)}^2 \le Ch^\beta.
\ee
Considering the
remaining quantity yields
\bbb
&&\l\|Q^{1/2}Ae^{(h-\xi_2)A}Dg[e^{\xi_2A}u]e^{\xi_2A}u\r\|_{L_2(H)}^2\nnn\\
&&~~~~~ \le \l\|Q^{1/2}(-A)^{(\beta-1)/2}\r\|_{L_2(H)}^2\l\|(-A)^{-\beta/2}e^{(h-\xi_2)A}\r\|_{L(H)}^2\l\|Dg[e^{\xi_2A}]e^{\xi_2A}u\r\|_{L_2(H)}^2\nnn\\
&&~~~~~ \le Ch^{\beta}\l\|Dg[e^{\xi_2A}u]e^{\xi_2A}u\r\|_{L_2(H)}^2
~ \le ~ Ch^\beta.\label{bnd22}
\eee
Combining \R{bnd11} and \R{bnd22} with \R{fbnd} yields
\bb{bnd3}
\mathbb{E}\|(T-S)(u)\|^2 \le Ch^{2+\beta} + \mathbb{E}\|R(u)\|^2.
\ee

The desired result follows by applying the bound in Lemma 5 to \R{bnd3}. \qed

\end{proof}


\begin{lemma}
Assume that $u\in\mbox{Dom}(A).$ Then 
$$\mathbb{E}\|R(u)\|^2 \le Ch^{2+\beta},$$
where $\beta\in(0,1]$ is defined in \R{normcond}.
\end{lemma}

\begin{proof}
We now demonstrate that all terms in $R(u)$ have the expected error bounds. Recalling $R(u),$ we have
\bbb
\mathbb{E}\|R(u)\|^2 &\le & 4\mathbb{E}\|R_T(u)\|^2 + 4\mathbb{E}\|R_1(u)+ 4\mathbb{E}\|R_2(u)\|^2 + 4\mathbb{E}\|R_3(u)\|^2\label{rr1}.
\eee
Let us estimate each of the terms in \R{rr1} individually. First, we observe that
\bbb
\mathbb{E}\|R_T(u)\|^2 & \le & \mathbb{E}\l\|\int_0^h\int_0^1 e^{(h-s)A}Df[\xi(y;s)](s\phi(s)+s\psi(s))\,dy\,ds\r\|^2\nnn\\
&&~~~~~~~~~~ + \mathbb{E}\l\|\int_0^h\int_0^1 e^{(h-y)A}Dg[\xi(y;s](s\phi(s)+s\psi(s))\,dy\,dW(s)\r\|_{L_2^0(H)}^2\nnn\\
& \le & \int_0^h\int_0^1 \mathbb{E}\l\|Df[\xi(y;s)](s\phi(s) + s\psi(s))\r\|^2\,dy\,ds, \nnn\\
&&~~~~~~~~~~+ Ch^{\beta-1}\int_0^h\int_0^1 \mathbb{E}\l\|Dg[\xi(y;s)](s\phi(s)+s\psi(s))\r\|_{L_2(H)}^2\,dy\,ds.\nnn
\eee
and by recalling that $Df$ and $Dg$ are uniformly bounded in $H,$ we obtain
\bbb
\mathbb{E}\|R_T(u)\|^2 & \le & C\int_0^h\int_0^1 \mathbb{E}\l\|s\phi(s)+s\psi(s)\r\|^2\,dy\,ds\nnn\\
&&~~~~~~~~~~ + Ch^{\beta-1}\int_0^h \int_0^1\mathbb{E}\l\|s\phi(s)+s\psi(s)\r\|_{L(H)}^2\,dy\,ds\nnn\\
& \le &  Ch^{2+\beta}.\label{ttra}
\eee

Recall \R{r1} and \R{r4}. Due to the fact that $Df$ is uniformly bounded, it is straightforward to show that
\bb{r1a}
\mathbb{E}\|R_1(u)\|^2 \le Ch^4\quad\mbox{and}\quad
\mathbb{E}\|R_3(u)\|^2 \le Ch^3.
\ee
Finally, according to \R{r2}, by invoking Assumption 6 we have
\bbb
&&\mathbb{E}\|R_2(u)\|^2\nnn\\
&&~~~~~ \le 2\int_0^h\int_0^s \mathbb{E}\l\|e^{hA}D^2g[u+y(e^{\Delta W(s)g}u - u)](e^{\Delta W(s)g}u - u)^2\r\|_{L_2^0(H)}^2\,dy\,ds\nnn\\
&&~~~~~~~~~~~~~~ +2\int_0^h \mathbb{E}\l\|e^{hA}Dg[u](e^{\Delta W(s)g}u-u)\r\|_{L_2^0(H)}^2\,ds.\nnn\\
&&~~~~~ \le Ch^{\beta-1}\int_0^h\l[\int_0^s \mathbb{E}\l\|e^{\Delta W(s)g}u - u\r\|^2\,dy + \mathbb{E}\l\|e^{\Delta W(s)g}u-u\r\|^2\r]\,ds.\nnn
\eee
By employing Lemma 6 in the above inequality, we obtain
\bb{r2a}
\mathbb{E}\|R_2(u)\|^2 \le Ch^{2+\beta}.
\ee


A combination of \R{ttra}-\R{r2a} yields our anticipated error bound. \qed


\end{proof}

Continuing, we may state the following estimate.

\begin{lemma}
Let $0<s<T.$ Then, for $u\in H,$ we have
$$\mathbb{E}\l\|e^{\Delta W(s)g}u - u\r\|^2 \le Cs.$$
\end{lemma}

\begin{proof}
By recalling \R{ppf1}, we see that
$$e^{\Delta W(s)g}u = u + \int_0^s g(e^{\Delta W(y)g}u)\,dW(y).$$
Thus, by Lemma 2, we have
\bbbb
\mathbb{E}\l\|e^{\Delta W(s)g}u - u\r\|^2 & = & \mathbb{E}\l\|\int_0^s g(e^{\Delta W(y)g}u)\,dW(y)\r\|^2\\
& = & \int_0^s \mathbb{E}\l\|g(e^{\Delta W(y)g}u)\r\|_{L_2^0(H)}^2\,dy
~ \le ~ Cs,
\eeee
which completes our proof.  \qed
\end{proof}

\section{Algorithmic Convergence}
\label{sec:5}

We now state our main result.
\begin{theorem}
Let $u_n = S^n(u_0),$ as defined in \R{s1}, be an approximation to the solution $u(nh) = T^n(u_0)$ 
of \R{e1}-\R{e2}. If $u_0\in\mbox{Dom}(A),$ then for $h$ sufficiently small we have
$$\mathbb{E}\|(S^n-T^n)(u_0)\|^2 \le Ch^{\beta},$$
where $\beta\in(0,1]$ is given in Assumption 4.   
\end{theorem}

\begin{proof}
Recall \R{tt}. It follows immediately that
\bbbb
T^n(u_0) &=& e^{nhA}u_0 + \int_0^{nh}e^{(nh-s)A}f(u(s))\,ds\\
&&~~~~~~~~~~~~~~~ + \int_0^{nh}e^{(nh-s)A}g(u(s))\,dW(s),\quad \mathbb{P}-a.s.
\eeee
We now have the following representation of the difference
\bb{diff1}
(S^n-T^n)(u_0) = \sum_{j=0}^{n-1} \l(S^{n-j}T^j - S^{n-j-1}T^{j+1}\r)(u_0).
\ee
By taking the norm and expectation of \R{diff1}, we observe that
\bbb
\mathbb{E}\|(S^n-T^n)(u_0)\|^2 & = & \mathbb{E}\l\|\sum_{j=0}^{n-1}\l(S^{n-j}T^j - S^{n-j-1}T^{j+1}\r)(u_0)\r\|^2\nnn\\
& \le & (n-1)\sum_{j=0}^{n-1}\mathbb{E}\l[L[S^{n-j-1}]\r]^2\mathbb{E}\|(S-T)(T^j(u_0))\|^2.\label{diff2}
\eee
If $u_0\in\mbox{Dom}(A),$ then it follows that $T^j(u_0)\in\mbox{Dom}(A),\ 0\le j\le n-1,$ due to Assumption 5. Therefore, we have
\bb{diffbnd1}
\mathbb{E}\|(S-T)(T^j(u_0))\|^2 \le Ch^{2+\beta},
\ee
for $0\le j\le n-1.$ Recall Lemma 3. We find that
\bb{diffbnd2}
\mathbb{E}\l[L[S^{n-j-1}]\r]^2 \le \mathbb{E}\l[L[S]\r]^{2(n-j-1)} \le e^{2(n-j-1)hC},
\ee
where $C$ is independent of $h,$ $n,$ and $j.$ Combining \R{diffbnd1} and \R{diffbnd2} gives
$$
\mathbb{E}\|(S^n-T^n)(u_0)\|^2 \le (n-1)\sum_{j=0}^{n-1} e^{2(n-j-1)hC}Ch^{2+\beta} \le Ch^\beta.
$$
\qed
\end{proof}

From Theorem 2, we see that the maximal mean square convergence rate is given by $\beta/2.$ Since $\beta\in (0,1],$ it follows that the maximal convergence rate is $1/2.$ Such a convergence rate is recovered when \R{e1}-\R{e2} is driven by trace class noise.





\bibliographystyle{spmpsci}
\bibliography{Springer_Bib}

\end{document}